\documentclass{amsart}

\usepackage{amssymb}
\usepackage{color}
\usepackage{amsxtra} 
\usepackage{mathrsfs} 
\usepackage{yfonts}

\newtheorem{theorem}{Theorem}

\newtheorem*{theorem*}{Main Theorem} 
\newtheorem*{theorem**}{Theorem}

\theoremstyle{definition}
\newtheorem{definition}[theorem]{Definition}

\theoremstyle{remark}
\newtheorem{remark}[theorem]{Remark} 
\newtheorem*{remark*}{Remark}

\numberwithin{equation}{section}
\numberwithin{theorem}{section} 

\begin{document}

\title[A mathematical consideration of vortex thinning  in 2D turbulence 
]
{A mathematical consideraion of vortex thinning in 2D turbulence
}

\author{Tsuyoshi Yoneda}
\address{Graduate School of Mathematical Sciences, University of Tokyo, Komaba 3-8-1 Meguro, Tokyo 153-8914, Japan} 
\email{yoneda@ms.u-tokyo.ac.jp}


\date{\today} 



\begin{abstract} 
In two dimensional turbulence, vortex thinning process is one of the attractive mechanism to explain inverse energy cascade in terms of vortex dynamics.
By direct numerical simulation to the two-dimensional Navier-Stokes equations with small-scale forcing and large-scale damping, Xiao-Wan-Chen-Eyink (2009) found an evidence that inverse energy cascade may proceed 
with the vortex thinning mechanism. 
The aim of this paper is to analyze the vortex-thinning mechanism mathematically (using the incompressible Euler equations), and 
give a mathematical evidence that large-scale vorticity gains energy from small-scale vorticity
due to the vortex-thinning process.
\end{abstract} 

\maketitle

\section{Introduction} \label{sec:Intro} 
``Vortex thinning"  is  one of the most important mechanism for   two-dimensional turbulence.
In \cite{XWCE},
Xiao-Wan-Chen-Eyink investigated inverse energy cascade in steady-state two-dimensional turbulence by direct numerical simulation of the two-dimensional Navier-Stokes equations with small scale forcing and  large scale damping. 
The two dimensional Navier-Stokes equations are described as follows:

%
\begin{align} \nonumber 
&\partial_tu -\nu\Delta u+ u\cdot\nabla u = - \nabla p+f, 
\qquad 
t \geq 0, \; x \in 
[-\pi,\pi)^2
\\ 
\label{E} 
&\mathrm{div}\, u = 0 
\\ 
\nonumber 
&u(0) = u_0 
\end{align} 
where $u=u(t,x)=(u_1(x,t),u_2(x,t))$, $p=p(x,t)$ and  $f=f_1(x_1,x_2),f_2(x_1,x_2))$ denote the velocity field, the pressure function of the fluid and the external force respectively. The case $\nu=0$, we call the Euler equations.
In their numerical work, they used an alternative  equations (adding a damping term),
and found strong evidence that inverse energy cascade may proceed with vortex thinning mechanism.
According to their evidence, there is a tensile turbulent stress in directions parallel to the isolines of small-scale vorticity.  
In their analysis,  the following large-scale strain tensor was  introduced:
\begin{equation*}
\bar S_\ell:=\frac{1}{2}\left[(\nabla \bar u_\ell)+(\nabla \bar u_\ell)^{T}\right],
\end{equation*}
where $\bar u_\ell$ is large-scale velocity defined by
\begin{equation*}
\bar u_\ell(x,t):=\int_{\mathbb{R}^2}G_\ell(r)u(x+r,t)dr
\end{equation*}
with $G_\ell(r)=\ell^{-2}G(r/\ell)$ and $G(r)=\sqrt{6/\pi}\exp(-6|r|^2)$. 
From $\bar S_\ell$, we can define a quantity of ``deformation work", the rate of work being done locally by the large-scale strain as it acts against the small-scale  stress. 
Thus the small-scale circular vortex will be stretched into elliptical shape.
According to Kelvin's theorem of conservation of circulation,
the magnitude of the velocity around the vortex decrease, and then its 
energy should be reduced. The energy lost by the small-scale vortex is transferred to the large scale.
In fact, the Reynolds stress created by the thinned vortex is primarily along the stretching direction,
positive. Hence, the deviatoric stress is positively aligned with the large-scale strain (this energy transfer mechanism can be justified by (3.4) and (3.11) in \cite{XWCE}). 
It does negative work, namely, negative eddy viscosity.  
For more detail of their argument, see \cite[Section 3 and Section 5]{XWCE}. 

The aim of this  paper is to analyze  such vortex-thinning mechanism   mathematically (without ``eddy viscosity"), and give a mathematical evidence that  large-scale vorticity gains  energy from small-scale vorticity due to the vortex-thinning process.
In our consideration, for the sake of simplicity, we 
 neglect the 
viscosity and the forcing terms. 
Thus
we  work with 
 the inviscid vorticity equations (equivalently, the Euler equations) in the whole space:
%
\begin{align} \label{vorticity equation} 
&\partial_t\omega + u{\cdot}\nabla \omega = 0, 
\qquad\qquad\qquad\quad 
t \geq 0, \; x \in \mathbb{R}^2 
\\  \nonumber
&\omega(0) = \omega_0\quad \text{with}\quad \int_{\mathbb{R}^2}\omega_0(x)dx=0, 
\end{align} 
where $u = \nabla^\perp \Delta^{-1} \omega$ and the vorticity $\omega$ is defined as $\omega = \partial_1 u_2 - \partial_2 u_1$.
Throughout this paper we aways handle smooth initial vorticity with compact support, thus
the corresponding  smooth solution always exists globally in time. Since the mean-value of the initial vorticity is zero, we see that the initial energy $\|u(0)\|_{L^2}$ is always finite.
Note that the first rigorous results for the existence theorem to \eqref{vorticity equation} were proved in the framework of H\"older spaces by Gyunter \cite{Gu}, Lichtenstein \cite{Li} and Wolibner \cite{Wo}. 
More refined results using a similar functional setting were obtained subsequently by 
Kato \cite{Ka,Ka2}, Swann \cite{Sw}, Bardos-Frisch \cite{BF}, Ebin \cite{Eb}, Chemin \cite{Ch}, 
Constantin \cite{Co1} and Majda-Bertozzi \cite{MB} among others. 
Here we just refer the Kato's existence theorem.
Let $H^s$ ($s\geq 0$) be  inhomogeneous Sobolev spaces.
\begin{theorem}
(\cite{Ka2})
Let $s>1$ and $\omega_0\in H^s(\mathbb{R}^2)$. Then there is a unique solution $\omega\in C([0,\infty); H^s(\mathbb{R}^2))$
of \eqref{vorticity equation}.
\end{theorem}
We have to mention that  the relation between the vortex-thinning process and palinstrophy ($H^1$-norm of the vorticity) has already been studied. 
In 
 \cite[Section 6.3]{AP} (see also \cite{MSMOM}),
they examined 2D vorticity equations with odd (in both $x_1$ and $x_2$) type of  initial vorticities
and measured these palinstrophy.
From their argument we know  
 that  palinstrophy 
is one of the key to see  2D turbulence. 
On the other hand, 
some of mathematicians have showed that 
there is an initial vorticity in $H^1$ such that the value of  
$\|\nabla \omega(t)\|_{L^2}$ (palinstrophy) instantaneously blows up.
More precisely, 
 Bourgain-Li \cite{BL} and  Elgindi-Jeong \cite{EJ} constructed  solutions
to the 2D-Euler equations 
which exhibit norm inflation  in  $H^1$ 
(see also 
\cite{MY2}). 
We essentially use their construction of the initial vorticity.
Let us now define ``vortex thinning" mathematically. 
In order to give  it, we need to define   ``Lagrangian flow $\Phi$" as follows:
\begin{equation*}
\partial_t\Phi(t,x)=u(t,\Phi(t,x))\quad\text{with}\quad \Phi(0,x)=x\in
\mathbb{R}^2
\end{equation*}
Note that the vortex-thinning  is one of the  ``large deformation gradient" 
which is already described in \cite[Proposition 3.4]{BL}.
\begin{definition}
First we choose $M>1$ and fix it.
We call ``$\Phi$ has a vortex thinning effect  at a point $x\in\text{supp}\, \omega_0(\Phi^{-1})$ and a time $t$ with direction $T$ (unit vector)" if and only if 
\begin{equation*}
\partial_T(\Phi(x,t)\cdot T)\geq M. 
\end{equation*}
\end{definition}
We now set a large-scale vorticity $\omega^L_0$ and a small-scale vorticity $\omega^S_0$ such that 
$\omega_0^L,\omega_0^S\in C^\infty(\mathbb{R}^2)$,
$\text{supp}\,\omega_0^L\cap\text{supp}\,\omega_0^S=\emptyset$,
$|\omega_0^L|, |\omega_0^S|\lesssim1$, 
both of them are odd in $x_1$ and $x_2$
(we consider more general vorticity cases in the next section).
More precisely,
let us  give  the large-scale initial vorticity
in the polar coordinate $(r,\theta)$ 
as defined in \cite[(5)]{EJ} (see also \cite[(3.4)]{BL} and \cite{KS}):
\begin{equation}\label{initial vorticity}
\omega_0^L(r,\theta):=\chi_N(r)\psi(\theta),
\end{equation}
where $\chi$ and $\psi$ are smooth bump functions, namely,  
\begin{equation*}
\chi_N(r):=
\begin{cases}
1\quad\text{for}\quad r\in [N^{-1},N^{-1/2}]\\
0\quad\text{for}\quad r\not\in [N^{-1}/2,2N^{-1/2}]
\end{cases}
 \text{and}\  
\psi(\theta):=
\begin{cases}
1\quad\text{for}\quad \theta\in [\pi/4,\pi/3]\\
\frac{1}{2}\quad\text{for}\quad \theta\in [\pi/5,9\pi/24]\\
0\quad\text{for}\quad \theta\not\in [\pi/6,5\pi/12].
\end{cases}
\end{equation*}
We choose $\omega_0^S$ such that 
$\text{supp}\,\omega_0^S\subset B(N^{-1}):=\{x\in\mathbb{R}^2: |x|<N^{-1}\}$ 
and 
\begin{equation*}
\{(h,h): n^{-1}\leq h\leq n^{-1}(\log N)^{K\tau^*}\}\subset \text{supp}\,\omega_0^S
\end{equation*}
for some $n>N$.
The meaning of the constants $K$ and $\tau^*$ are the same as in \cite{EJ}.
In this case $H^1$-norm (palinstorophy) may be large provided by large $N$.
In \cite{BL} they figured out that this kind of  odd vorticity creates large  deformation gradient.
But they employed a contradiction argument and thus we cannot figure out what kind of deformation and where it occurs (see Proof of Proposition 3.4 in their paper).
Nevertheless, with the help of Elgindi-Jeong's argument in \cite{EJ}, we can clarify it.
The main result is as follows:

\begin{theorem}
For any $M>1$, there is $N_0>0$ and $\tau^*>0$ such that if $N>N_0$, 
then there is $t\in (0, \tau^*\log\log N/\log N]$,
at least either of the following two cases must occurs.
\begin{itemize}
\item
Small-scale vortex-thinning:
\begin{equation*}
 \partial_{e_2}(\Phi(x,t)\cdot e_2) \geq M^{1/2}
\end{equation*}
for some
$
x\in \text{supp}\,\omega_0^S
$,
where
$e_2=(0,1)$.
\item 
Large-scale vortex-thinning (but it stretches only tail part):
\begin{equation*}
 \partial_T(\Phi(x,t)\cdot T) \geq M^{1/2}
\end{equation*}
for some 
$
x\in \text{supp}\,\omega_0^L
\cap D
$
and
$T=\frac{1}{\sqrt{2}}(1,1)$,
where 
\begin{equation*}
D=B\left(\alpha\tau^*N^{-1/2}\frac{\log\log N}{\log N}\right)
\end{equation*}
with sufficiently large positive constant $\alpha>0$.
(In this case we can rephrase  $\omega_0^L\cdot \chi_{D}$ as the small-scale vorticity,
while $\omega_0^L\cdot (1-\chi_{D})$ as the large-scale vorticity parts.)
\end{itemize}

\end{theorem}
\begin{proof}
The idea is to use Elgindi-Jeong's argument \cite{EJ} with minor modifications.
By 
 (10) in \cite{EJ}, any line segment 
\begin{equation*}
\{(r,\theta_0): N^{-1}\leq r\leq N^{-1/2}\}
\end{equation*}
evolve in a way that it intersects each circle 
$
\{r=r_0\}
$
for
$N^{-5/6}\leq r_0\leq N^{-4/6}$.
Recall that 
\begin{equation*}
I(t,r_0):=\{0\leq \theta\leq \pi/2: (r_0,\theta)\in R(t)\}
\end{equation*}
with
\begin{equation*}
R(t):=\Phi(t,V)\cap A,
\end{equation*}
\begin{equation*}
A:=\{r:N^{-\frac{5}{6}}\leq r\leq N^{-\frac{4}{6}}\}
\quad\text{and}\quad
V:=\{(r,\theta):\omega_0^L(r,\theta)\geq 1/2\}.
\end{equation*}

We now  consider two cases as in \cite{EJ}:
\begin{itemize}
\item
The case I:
there exists a time moment $t_0$ such that for 
more than half of $r_0\in [N^{-5/6},N^{-4/6}]$ in the Haar measure, we have  $|I(t_0, r_0)|\leq M^{-1}$.
\item
The case II:
for all time $t_0\in [0,\tau^*\log\log N/\log N]$, at least half of $r_0\in [N^{-5/6},N^{-4/6}]$ in the Haar measure, we have  $|I(t_0, r_0)|\geq M^{-1}$.
\end{itemize}


\vspace{0.2cm}

Let us now first consider the case I.
In this case, diagonal direction of vortex thinning effect to $\omega^L$ itself occurs.
Let
\begin{equation*}
\partial\text{cone}^+:=
\{(r,\theta):r>0, \theta=9\pi/24\},
\quad
\partial\text{cone}^-:=
\{(r,\theta):r>0, \theta=\pi/5\}
\end{equation*}
and $\partial\text{cone}=\partial\text{cone}^+\cup\partial \text{cone}^-$.
Let $n=(1/\sqrt 2)(-1,1)$. In this case there exists a time moment $t_0$ and there are two points $x_1 \in\partial\text{cone}^+\cap \text{supp}\, \omega^L_0$ and $x_2 \in\partial\text{cone}^-\cap \text{supp}\, \omega^L_0$ such that  
\begin{equation*}
|(\Phi(x_1,t_0)-\Phi(x_2,t_0))\cdot n|\approx
|\Phi(x_1,t_0)-\Phi(x_2,t_0)|\lesssim M^{-1}N^{-m}
\end{equation*}
with
\begin{equation*}
\Phi(x_1,t_0)\cap \partial B(N^{-m})\not=\emptyset\quad\text{and}\quad \Phi(x_2,t_0)\cap \partial B(N^{-m})\not=\emptyset
\end{equation*}
for some $m\geq 9/12$ (this $9/12$ comes from ``more than half of $[N^{-5/6},N^{-4/6}]$ in the Haar measure").
If both $x_1$ and $x_2$ satisfy
\begin{equation*}
x_1,x_2\in \partial \text{cone}\cap\left(B(N^{-1/2})\setminus B(N^{-m}M^{-1/2})\right),
\end{equation*}
then we can estimate the distance between $x_1$ and $x_2$ as 
\begin{equation*}
|(x_1-x_2)\cdot n|\gtrsim N^{-m}M^{-1/2},
\end{equation*}
where $n=(1/\sqrt 2)(-1,1)$.
Thus by the mean value theorem,
 there is a point $y\in \text{supp}\, \omega_0(\Phi^{-1})$
 such that 
\begin{equation*}
|\partial_n(\Phi^{-1}(y,t_0)\cdot n)|\gtrsim M^{1/2}.
\end{equation*}
By the inverse function theorem with volume preserving, we have 
\begin{equation}\label{first inequality}
|\partial_\tau(\Phi(x,t_0)\cdot \tau)|\gtrsim M^{1/2}
\end{equation}
with $x=\Phi^{-1}(y,t_0)$.
This is the desired estimate.
If 
at most one point $x_1$ satisfies 
\begin{equation*}
x_1\in (B(N^{-1/2})\setminus B(N^{-m}M^{-1/2}))\cap \partial\text{cone}
\end{equation*}
for some $m>9/12$,
then the other point must be
\begin{equation*}
x_2\in B(N^{-m}M^{-1/2})\cap\partial\text{cone}.
\end{equation*}
In this case  we choose a third point $x_3$ to be 
\begin{equation*}
x_3\in B(N^{-1})\subset B(\epsilon N^{-m})
\end{equation*}
for sufficiently small $\epsilon>0$ (we choose $\epsilon$ so that 
$\epsilon>N^{-1+m}(\log N)^{C\tau}$).
Then by (10) in \cite{EJ}, we see
\begin{equation*}
\Phi(t_0,x_3)\in B(N^{-1}(\log N)^{C\tau})\subset B(\epsilon N^{-m})
\end{equation*}
for large $N>0$. Note that  the constants $C$ and $\tau$ have the same meaning in \cite{EJ}.
Thus we have 
\begin{equation*}
|(x_2-x_3)\cdot \tau|\lesssim N^{-m}(M^{-1/2}-\epsilon)
\end{equation*}
and 
\begin{equation*}
|(\Phi(t_0,x_2)-\Phi(t_0,x_3))\cdot \tau|\gtrsim N^{-m}(1-\epsilon).
\end{equation*}
By the mean-value theorem, we finally have 
\begin{equation}\label{second inequality}
|\partial_\tau(\Phi\cdot \tau)|\gtrsim M^{1/2}(1-\epsilon).
\end{equation}
This is  the desired estimate.

\begin{remark}
At least, $|I(t_0,r_0)|\leq M^{-1}$ never occur in
\begin{equation*}
r_0\in \left[\alpha \tau^*N^{-1/2}\frac{\log\log N}{\log N}, N^{-1/2}\right]
\end{equation*}
for sufficiently large $\alpha>0$.
In fact, since $\omega(t,x)=\omega_0(\Phi^{-1}(t,x))$, $\text{supp}\, \omega_0\subset B(N^{-1/2})$, $|\omega_0(x)|\leq 1$, we have the following a priori velocity estimate:
\begin{eqnarray*}
|u(t_0,x)|&=&|\nabla^{\perp}\Delta^{-1}\omega(t_0,x)|\leq |\nabla^{\perp}\Delta^{-1}\omega_0(\eta^{-1}(t_0,x))|\\
&\lesssim&
\sup_{\stackrel{\tilde\Phi\in (C^\infty(\mathbb{R}^2))^2}{\ \det D\tilde \Phi=1}}
\int_{
[0,\infty)^2
}\omega_0(\tilde \Phi(y))\frac{dy}{|y-x|}
\lesssim
\int_{B(N^{-1/2})}\frac{1}{|x|}dx
\lesssim N^{-1/2}.
\end{eqnarray*}
Let 
\begin{eqnarray*}
& &
x_1\in\partial\text{cone}^+\cap \left(B(N^{-1/2})\setminus B(\alpha\tau^*N^{-1/2}\log\log N/\log N)\right),\\
& &
x_2\in\partial\text{cone}^-\cap \left(B(N^{-1/2})\setminus B(\alpha\tau^*N^{-1/2}\log\log N/\log N)\right)
\end{eqnarray*}
with sufficiently large constant $\alpha>0$.
By the above a priori velocity estimate, we have 
\begin{equation}\label{a priori velocity estimate}
|\Phi(t_0,x_1)-\Phi(t_0,x_2)|\geq |x_1-x_2|-2N^{-1/2}t\gtrsim |x_1-x_2|
\end{equation}
for $t_0\in [0, \tau\log\log N/\log N)$.
Thus $|I(t_0,r_0)|\geq M^{-1}$ provided by sufficiently large $\alpha$.

\end{remark}

\vspace{0.2cm}

Next we consider the case II.
In this case $x_2$-direction of vortex thinning effect to the small-scale vortex occurs,
while
large-scale vorticity $\omega_0^L$ creates large scale strain. 
Let us choose two points $x_1$, $x_2\in \text{supp}\, \omega^S_0$ as  (the constants $K$ and $\tau^*$ have the same meaning in \cite{EJ})
\begin{equation*}
x_1=(n^{-1},n^{-1})\quad\text{and}\quad x_2=((\log N)^{K\tau^*}n^{-1},(\log N)^{K\tau^*}n^{-1})
\end{equation*}
for some $n>N$.
Now we recall  Zlatos's velocity estimate \cite{Z} (just extend it to the whole space case):
\begin{theorem}
Let $\omega(t,\cdot)$ be odd in $x_1$ and $x_2$. Then for $x\in[0,1/2)^2$, we have 
\begin{equation*}
u^i(t,x)/x_i=(-1)^iQ(t,x)+B_i(t,x)
\end{equation*}
with 
\begin{equation*}
Q(t,x)=\frac{4}{\pi}\int_{[2x_1,\infty)\times [2x_2,\infty)}\frac{y_1y_2}{|y|^2}\omega(t,y)dy
\end{equation*}
and $|B_i|\leq C\|\omega\|_\infty(1+\log(1+x_{3-i}/x_i)$ for $i=1,2$.
\end{theorem}
By the same argument as in \cite{EJ} (the constants $c_M$ and $C_M$ have the same meaning in there), we have
\begin{equation*}
Q(t, x_2)\geq c_M\log N
\end{equation*}
and the $B_i(t,x)$-term can be neglected.
Thus we have 
\begin{equation*}
\partial_t\Phi_2(t, x_2)\geq C_M\log N \Phi_2(t,x_2).
\end{equation*}
Let $e_2=(0,1)$. Then we see that 
\begin{eqnarray*}
\frac{(\Phi(x_2)-\Phi(x_1))\cdot e_2}{(x_2-x_1)\cdot e_2}
&\geq &
\frac{(\log N)^{C_M}(\log N)^{K\tau^*}-(\log N)^{C\tau}}{(\log N)^{K\tau^*}-1}\\
&=&
\frac{(\log N)^{C_M}-(\log N)^{C\tau-K\tau^*}}{1-(\log N)^{-K\tau^*}}.
\\
\end{eqnarray*}
In this case we choose 
$K$
such that $C-K<0$
and choose sufficiently large $N$ so that  $(\log N)^{C_M}\geq  M$.
By the mean-value theorem, we have 
\begin{equation}\label{third inequality}
\partial_{e_2}(\Phi\cdot e_2)\approx (\log N)^{C_M}\gtrsim M.
\end{equation}

\end{proof}


\section{General vorticity setting}\label{general setting}

In this section we  extend the vortex-thinning mechanism to general vorticity cases.
Since we need to require finite energy, mean-zero vorticity condition
$\int\omega_0=0$ should be required. Otherwise, the energy becomes infinite due to the slowly decaying velocity.  
Let $\partial_t \Phi(x,t)=u(\Phi(x,t),t)$ be a solution to the Euler equations (vorticity equations) \eqref{vorticity equation} with the initial vorticity $\text{rot}\,u_0=\omega^S_0+\omega^L_0$, and let $\partial_t\Psi(x,t)=(u+v)(\Psi(x,t),t)$ also be a solution to the Euler equations \eqref{vorticity equation} with the initial vorticity  $\text{rot}\,u_0+\text{rot}\,v_0$ (in this case $v_0$ is a perturbation, and assume $\omega^P_0:=\text{rot}\, v_0$ has  compact support).
By  \cite[Lemma 4.1]{BL}, we have 
\begin{eqnarray*}
& &
\sup_{0<t<1}
|\partial_T(\Phi\cdot T)-\partial_T(\Psi\cdot T)|
\leq 
\sup_{0<t<1}|D\Phi(t,x)-D\Psi(t,x)|\\
&\leq& \left(\sup_{0<t<1}\|v(t)\|_\infty+\sup_{0<t<1}\|\nabla v(t)\|_\infty \right)\exp\left\{\sup_{0<t<1}\|\nabla u(t)\|_\infty\right\}.
\end{eqnarray*}

By the Sobolev embedding,
\begin{equation}\label{Sobolev embedding}
\|v(t)\|_{\infty}+\|\nabla v(t)\|_\infty\leq C_1(s)\|v(t)\|_{H^s}\quad\text{for}\quad s>2,
\end{equation}
where $C_1$ is a 
  positive constant  satisfying $C_1(s)\to\infty$ as $s\to 2$.
Moreover, by the continuity on initial velocity  in $H^s$ $(s>2)$, we have 
\begin{equation*}
\sup_{0<t<1}\|v(t)\|_s\leq C_2(s,\|u_0\|_s,\|u_0+v_0\|_s)\|v_0\|_s,
\end{equation*}
where $C_2$ is a
 positive constant satisfying $C_2(s,\|u_0\|_s,\|u_0+v_0\|_s)\to \infty$ as $s\to 2$.
 Continuity of the solution map for the Euler equations in Sobolev spaces 
$H^{s}$ 
for 
 $s > 2$ is of course well known 
(see e.g., Ebin-Marsden \cite{EbMa}, Kato-Lai \cite{KL}, Kato-Ponce \cite{KP-Duke} and also \cite{MY2}).
Thus  if the perturbation $v_0$ is controlled as 
\begin{equation*}
C_1C_2\|v_0\|_s\leq C_3 M^{1/2}\quad\text{for some}\quad s>2
\end{equation*}
with some constant $C_3>0$ determined by \eqref{first inequality}, \eqref{second inequality} and \eqref{third inequality},
then we get the same vortex thinning mechanism to the initial velocity $u_0+v_0$.
This means that a distorted symmetry case (measured in $H^s$) also keeps the 
vortex-thinning process in a short time interval. 

Vorticity far from the origin (we call ``remainder part") does not strongly affect to the vortex-thinning process which is occurring near the origin.
In this case we just apply ``gluing the patches argument", Lemma 5.2  in \cite{BL}.
Let
\begin{align} \label{euler1} 
&\partial_t\tilde \omega + \tilde u{\cdot}\nabla \tilde \omega = 0, 
\qquad\qquad\qquad\quad 
t \geq 0, \; x \in \mathbb{R}^2 
\\  \nonumber
&\tilde \omega(0) =f\quad\text{with}\quad f= \omega_0^L+\omega^S_0+\omega^P_0,  \\
\nonumber
&
\partial_t\tilde \Phi=\tilde u(\tilde \Phi),
\end{align} 
where $\tilde u = \nabla^\perp \Delta^{-1} \tilde \omega$ and 
 \begin{align} \label{euler2} 
&\partial_t \omega +  u{\cdot}\nabla \omega = 0, 
\qquad\qquad\qquad\quad 
t \geq 0, \; x \in \mathbb{R}^2 
\\  \nonumber
&\omega(0) =f+g \quad\text{with}\quad g= \omega^R_0,\\
\nonumber
& 
\partial_t\Phi= u(\Phi),
\end{align} 
where $u = \nabla^\perp \Delta^{-1}\omega$. Here we   need to assume 
$\omega_0^R$ is in $L^p\cap L^1$ $(p>2)$ with $\|\omega_0^R\|_{L^p}+\|\omega_0^R\|_{L^1}\lesssim 1$
In this case we have 
\begin{eqnarray*}
|u(t,x)|&=&|\nabla^{\perp}\Delta^{-1}\omega(t,x)|\leq |\nabla^{\perp}\Delta^{-1}(f+g)(\Phi^{-1}(t,x))|\\
&\leq&
\sup_{\stackrel{\tilde\Psi\in (C^\infty(\mathbb{R}^2))^2}{\ \det D\tilde \Psi=1}}
\int_{\mathbb{R}^2}(f+g)(\tilde \Psi(y))\frac{dy}{|y-x|}
\leq \beta
\end{eqnarray*}
for some positive constant $\beta>0$.
By the above a priori velocity estimate, it is reasonable to assume 
\begin{equation*}
d(\text{supp}\, f,\text{supp}\, g)\geq \beta.
\end{equation*}
In this case, the supports of $\tilde \omega(t)$ and $\omega^R(t)$ are always disjoint in $t\in [0,1]$.
But we moreover need to assume that 
\begin{equation*}
d(\text{supp}\, f,\text{supp}\, g)\geq R_\epsilon\, (>\beta),
\end{equation*}
where $R_\epsilon$ is already defined in Lemma 5.2 in \cite{BL}.
With a minor modification of the proof of Lemma 5.2 in \cite{BL} (see also Remark \ref{adding remainder case} in this paper),
we immediately have the following:
\begin{theorem}
Let $s>1$ be fixed.
Also let $D=\{x: d(x,\text{supp}\,f)<\beta\}$ and $\omega_f:=\chi_D\omega$.
For any sufficiently small $\epsilon>0$, we have 
\begin{equation*}
\sup_{0\leq t\leq 1}\|\tilde \omega(t)-\omega_f(t)\|_{H^s}<\epsilon
\end{equation*}
provided that $R_\epsilon>0$ is sufficiently large.
\end{theorem}

By the above theorem with the Sobolev embedding \eqref{Sobolev embedding}, we can easily show that the initial vorticity $f+g$ also create the  vortex-thinning process.

\section{energy transfer from small scale vortex to large scale vortex}

In this section we give an evidence of  energy transfer from small-scale vortex to  large-scale vortex.
For the small-scale vortex $\omega^S$, 
we  assume a simple vortex-thinning process: $\omega^S(x,t)=\omega^S_0(Mt x_1,(Mt)^{-1} x_2)$ 
for the sake of simplicity. Let 
\begin{equation*}
u^L:=\nabla^{\perp}\Delta^{-1}\omega^L\quad\text{and}\quad u^S:=\nabla^{\perp}\Delta^{-1}\omega^S.
\end{equation*}
We can measure the energies of each scale vortices.
In this section we also assume that  the mean value of the each scale vortices are zero, thus the  energies of the each vortices are also finite.
Also assume that $\text{supp}\, \omega_0^S\cap \text{supp}\, \omega_0^L=\emptyset$.
 In this case we just directly take the $L^2$-norm to $u^S$ and $u^L$.

\begin{theorem}
Let $\omega^S+\omega^L$ be a solution to \eqref{vorticity equation}
with initial vorticity $\omega_0^S+\omega_0^L$.
For the initial vorticity $\omega_0^S+\omega_0^L$ or $-\omega_0^S+\omega_0^L$, 
then we have the following energy  estimate:
\begin{equation*}
(\|u^L_0\|_{L^2}^2+\|u^S_0\|_{L^2}^2)^{1/2}-\|u^S(t)\|_{L^2}\leq \|u^L(t)\|_{L^2}.
\end{equation*}
Moreover, for fixed $t\in (0,1]$,
\begin{equation*}
\|u^S(t)\|_{L^2}\to 0 \quad\text{as}\quad M\to \infty.
\end{equation*}
These estimates are the evidence of  the energy-transfer mechanism.
\end{theorem}
\begin{proof}
If the velocity interaction is negative, namely, 
\begin{equation}\label{interaction of velocity}
\int u_0^L\cdot u_0^S<0,
\end{equation}
then we replace $\omega_0^S$ to $-\omega_0^S$.
In this case, the above integration \eqref{interaction of velocity} becomes positive.
By the enstrophy conservation, and the disjoint supports, we see 
\begin{equation*}
\|\omega^L(t)\|_{L^2}=\|\omega^L_0\|_{L^2}\quad\text{and}\quad
\|\omega^S(t)\|_{L^2}=\|\omega^S_0\|_{L^2}.
\end{equation*}
By the mean-zero vortex $\omega_0^S$, there is $\Omega_0\in C^\infty_c$ such that 
$\partial_1\Omega_0=\omega_0^S$.
By taking the Fourier transform, we have 
\begin{eqnarray*}
\|u^S_0((Mt)^{-1}\cdot, Mt\cdot)\|_{L^2}^2
&\leq& 
\int \frac{|\hat \omega_0^S(\xi_1,\xi_2)|^2}{((Mt\xi_1)^2+((Mt)^{-1}\xi_2)^2)^{1/2}}d\xi\\
&\leq&
 \frac{1}{Mt}\int|\hat \Omega_0(\xi_1,\xi_2)|^2d\xi
\to 0\quad\text{as}\quad M\to\infty.
\end{eqnarray*}
Also by the energy conservation to the incompressible Euler flow, we see
\begin{eqnarray*}
\|u^L_0\|_{L^2}^2+\|u^S_0\|_{L^2}^2&\leq &
\|u^L_0\|_{L^2}^2+2\int u^L_0\cdot u^S_0+\|u^S_0\|_{L^2}^2\\
&=&\|u^L_0+u^S_0\|_{L^2}^2=\|u^L(t)+u^S(t)\|_{L^2}^2\\
&\leq &
(\|u^L(t)\|+\|u^S(t)\|_{L^2})^2
\end{eqnarray*}
By the above estimate, the large scale vorticity  gains  energy from the small-scale vorticity.

\end{proof}

\begin{remark}\label{adding remainder case}
By the same argument as in the previous section, even if we add $g$ in \eqref{euler2}
to the initial vorticity, we can still get the energy-transfer process provided by sufficiently small $\epsilon>0$ and 
large $R_\epsilon$.
We use the same notations in Section \ref{general setting}.
Let $u_f=\nabla^{\perp}\Delta^{-1}\omega_f$, $u_g=\nabla^{\perp}\Delta^{-1}\omega_g$
and $\omega_g(t,x)=g(\Phi^{-1}(t,x))$.
Note that 
\begin{equation*}
\omega=\omega_f+\omega_g=f(\Phi^{-1})+g(\Phi^{-1})
\end{equation*}
and $u_f$ satisfies
\begin{equation*}
\partial_tu_f=(u_f\cdot \nabla)u_f+(u_g\cdot \nabla)u_f=-\nabla p_f
\end{equation*}
with some scalar function $p_f$.
Set $\eta=u_f-\tilde u$.
Then we have 
\begin{equation}\label{eta}
\partial_t\eta+(\eta\cdot\nabla)u_f+(\tilde u\cdot \nabla)\eta+(u_g\cdot\nabla)u_f=-\nabla p_\eta
\end{equation}
with some scalar function $p_\eta$.
If the supports of $\omega_g$ and $\omega_f$ are far from each other, then 
$\|(u_g\cdot \nabla)u_f\|_{L^2}$ is small enough.
By the usual well-posedness theorem with a commutator estimate,
 then we can control $\|u_f\|_{H^s}$ for $s=3,4,5\cdots$
(we use the initial vorticity $f+g$, and the norm in $H^s$ is independent of the distance between $f$ and $g$).
Thus we have 
\begin{equation*}
\|\nabla \omega_f\|_\infty\leq \|\omega_f\|_{H^{2}}\leq \|\omega_f+\omega_g\|_{H^{2}}
\lesssim \|f+g\|_{H^{2}}
\end{equation*}
and also 
\begin{equation*}
\|\nabla u_f\|_\infty\leq \|\omega_f\|_{H^{3}}\leq \|\omega_f+\omega_g\|_{H^{3}}
\lesssim \|f+g\|_{H^{3}}
\end{equation*}
for $t\in[0,1]$.
We  multiply $\eta$ to \eqref{eta} on both sides, integrate on $\mathbb{R}^2$ and with some algebra,  we obtain the following energy inequality:
\begin{equation*}
\partial_t\|\eta\|_2^2\lesssim \|\nabla u_f\|_\infty\|\eta\|_2^2+\|\eta\|_2\|(u_g\cdot \nabla)u_f\|_{L^2}.
\end{equation*}
Thus the energy $\|\eta\|_2$ is small if $\|(u_g\cdot \nabla)u_f\|_{L^2}$ is sufficiently small.
\end{remark}


\bibliographystyle{amsplain}

\begin{thebibliography}{10} 

\bibitem{AP}
D. Ayala and B. Protas,
\textit{Maximum palinstrophy growth in 2D incompressible flows},
preprint, arXiv:1305.7259v2.


\bibitem{BF} 
C. Bardos and U. Frisch, 
\textit{Finite-time regularity for bounded and unbounded ideal incompressible fluids using H\"older estimates}, 
Turbulence and Navier-Stokes equations (Proc. Conf., Univ. Paris-Sud, Orsay, 1975), 
Lecture Notes in Math., vol. \textbf{565}, Springer, Berlin 1976





\bibitem{BL} 
J. Bourgain and D. Li, 
\textit{Strong ill-posedness of the incompressible Euler equations in borderline Sobolev spaces}, 
Invent. math. \textbf{201}, (2015), 97-157; 
preprint arXiv:1307.7090 [math.AP]. 



\bibitem{Ch} 
J. Chemin, 
\textit{Perfect Incompressible Fluids}, 
Clarendon Press, Oxford 1998. 


\bibitem{Co1} 
P. Constantin, 
\textit{An Eulerian-Lagrangian approach for incompressible fluids: local theory}, 
J. Amer. Math. Soc. \textbf{14} (2001), 263-278. 



\bibitem{EbMa}
D. Ebin and J. Marsden, 
\textit{Groups of diffeomorphisms and the motion of an incompressible fluid}, 
Ann. Math. \textbf{92} (1970), 102-163.  

\bibitem{Eb} 
D. Ebin, 
\textit{A concise presentation of the Euler equations of hydrodynamics}, 
Comm. Partial Differential Equations \textbf{9} (1984), 539-559. 

\bibitem{EJ}
T. Elgindi, and I.-J. Jeong, 
\textit{Ill-posedness for the incompressible Euler equations in critical Sobolev spaces},
 arXiv:1603.07820.

 



\bibitem{Gu} 
N. Gyunter, 
\textit{On the motion of a fluid contained in a given moving vessel}, 
(Russian), Izvestia Akad. Nauk USSR, Ser. Phys. Math. \textbf{20} (1926), 1323-1348, 1503-1532; 
\textbf{21} (1927), 621-556, 735-756, 1139-1162; \textbf{22} (1928), 9-30. 


\bibitem{Ka} 
T. Kato, 
\textit{On classical solutions of the two-dimensional non-stationary Euler equation}, 
Arch. Ration. Mech. Anal. \textbf{25} (1967), 188-200. 

\bibitem{Ka2} 
T. Kato, 
\textit{Remarks on the Euler and Navier-Stokes equations in $\mathbb{R}^2$}, 
Proc. Sym. Pure Math., \textbf{45} (1986), 1-7. 


\bibitem{KL} 
T. Kato and C. Lai, 
\textit{Nonlinear evolution equations and the Euler flow}, 
J. Funct. Anal. \textbf{56} (1984), 15-28. 

\bibitem{KP-Duke}
T. Kato and G. Ponce,
\textit{On nonstationary flows of viscous and ideal fluids in $L^p_s(\mathbb{R}^2)$},
Duke Math. J. \textbf{55} (1987), 487-499.


\bibitem{KS}
A. Kiselev and V. \v Sver\'ak,
\textit{Small scale creation for solutions of the incompressible two-dimensional Euler equation},
Ann. of Math., \textbf{180} (2014) 1205-1220.



\bibitem{Li} 
L. Lichtenstein, 
\textit{Uber einige Existenzprobleme der Hydrodynamik}, 
Math. Zeit. \textbf{23} (1925), 89-154, 309-316; \textbf{26} (1927), 196-323; \textbf{28} (1928), 387-415; 
\textbf{32} (1930), 608-640. 

\bibitem{MB} 
A. Majda and A. Bertozzi, 
\textit{Vorticity and Incompressible Flow}, 
Cambridge University Press, Cambridge 2002. 


\bibitem{MSMOM}
W. H. Matthaeus, W. T. Stribling, D. Martinez, S. Oughton and David Montgomery,
\textit{Selective decay and coherent vortices in two-dimensional incompressible turbulence},
Phys. Rev. Lett. \textbf{66}  (1991),  2731-2734.


\bibitem{MY2} 
G. Misio{\l}ek and T. Yoneda, 
\textit{Continuity of the solution map of the Euler equations in H\"older spaces and weak norm inflation in Besov spaces}, 
preprint arXiv: 1601.01024 [math.AP] 



\bibitem{Sw} 
H. Swann, 
\textit{The existence and uniqueness of nonstationary ideal incompressible flow in bounded domains in $R_3$}, 
Trans. Amer. Math. Soc. \textbf{179} (1973), 167-180. 




\bibitem{Wo} 
W. Wolibner, 
\textit{Un theor\'eme sur l'existence du mouvement plan d'un fluide parfait, homog\'ene, 
incompressible, pendant un temps infiniment long}, 
Math. Z. \textbf{37} (1933), 698-726. 







\bibitem{XWCE}
Z. Xiao, M. Wan, S. Chen and G. L. Eyink,
\textit{Physical mechanism of the inverse energy cascade of two-dimensional turbulence: a numerical investigation},
J. Fluid Mech. \textbf{619} (2009), 1-44.


\bibitem{Z}
A. Zlato\v s,
\textit{Exponential growth of the vorticity gradient for the Euler equation on the 
torus},
Adv. Math., \textbf{268} (2015), 396-403


\end{thebibliography}

\end{document}